\documentclass{amsart}
\usepackage[left=1in,right=1in,top=1.2in,bottom=1in]{geometry}
\usepackage[utf8]{inputenc}
\usepackage{setspace}
\usepackage{amsmath,amssymb}
\usepackage{comment}
\usepackage{xcolor}
\usepackage{url}
\usepackage[colorlinks=true,allcolors=blue,backref=page]{hyperref}
\usepackage[noabbrev,capitalize,nameinlink]{cleveref}
\usepackage{bbm}
\usepackage{centernot}
\usepackage{extarrows}
\usepackage[shortlabels]{enumitem}
\allowdisplaybreaks

\newtheorem{thm}{Theorem}

\newtheorem{lem}[thm]{Lemma}

\newtheorem{prop}[thm]{Proposition}
\newtheorem{claim}[thm]{Claim}
\theoremstyle{remark}
\newtheorem{remark}{Remark}

\newcommand{\R}{\mathbb{R}}
\newcommand{\Z}{\mathbb{Z}}
\newcommand{\Prob}{\mathbb{P}}
\newcommand{\eps}{\varepsilon}
\newcommand{\1}{\mathbbm{1}}
\renewcommand{\epsilon}{\eps}
\renewcommand{\P}{\mathbb{P}}

\newcommand{\E}{\mathbb{E}}
\newcommand{\ind}{\mathbbm{1}}
\newcommand{\ring}{\mathring}
\newcommand{\lam}{\lambda}

\newcommand{\dg}{d_{G_\lam}}

\newcommand{\xlra}{\xleftrightarrow{G_\lam}}
\nocite{*}

\title{Lipschitz continuity of the time constant for continuum percolation}
\author{Karoline Dubin}
\address{University of Illinois, Chicago. Dept of Mathematics, Statistics and Computer science.}
\email{kdubin3@uic.edu}

\author{Christian Gorski}
\address{University of Washington. Dept of Mathematics.}
\email{cgorski1@uw.edu}

\begin{document}

\begin{abstract}


    We consider the Boolean model of continuum percolation, where points are placed in $\R^d$ by a Poisson point process and pairs of points with distance at most 1 are connected by an edge. 
    The time constant is the 
    limiting ratio
    of the chemical distance
    (i.e. graph distance) to the Euclidean 
    distance for pairs of distant connected points.
    Yao, Chen, and Guo \cite{YCG11} 
    established the existence of a time constant
    in the supercritical regime. We show that above the critical intensity, 
    the time constant is a Lipschitz continuous function
    of the intensity.
    The proof adapts
    a recent argument
    of Can, Nakajima, and Nguyen \cite{CNN23}
    to the continuous setting.
\end{abstract}
\maketitle

\section{Introduction}

Given a point set $\mathcal{X} \subset \R^d$, we denote by $G(\mathcal{X})$ the graph with vertex set $\mathcal{X}$ and edge set consisting of $\{x,y\} \subset \mathcal{X}$ such that $\|x-y\|\le 1$. Let $X_\lambda$ denote the homogeneous Poisson point process of intensity $\lambda$ on $\R^d$. We consider the infinite random geometric graph $G(X_\lambda)$, which we abbreviate by $G_\lam$.  
The critical intensity $\lam_c = \lam_c(d)$ is defined by 
\[
\lam_c = \inf \left\{\lam>0 : 
\Prob(G_\lam \mbox{ contains an infinite component)}
 > 0 \right \}.
\]
When $d \ge 2$, $0<\lam_c<\infty$, and if $\lam > \lam_c$ (the supercritical regime) then $G_\lam$ almost surely has a unique infinite component,
which we denote by $C_\infty$.

We define a natural metric on $X_\lam \subset G_\lam$ by taking the distance between two vertices to be the least number of edges in a path between them. 
For all $x,y \in G_\lam$, the \emph{chemical distance} between $x$ and $y$ is
\[
d_{G_\lam}(x,y) := \inf\{|\pi|: \pi \text{ is a path in } G_\lam \text{ from } x \text{ to } y\}
\]
where $|\pi|$ is the number of edges in $\pi$. 
Yao, Chen, and Guo  \cite{YCG11} 
prove the existence of a \emph{time constant} for the chemical distance in supercritical continuum percolation: for each $\lam > \lam_c(d)$, there exists a constant $\mu_\lam(d) \ge 1$ such that 
\[
\lim_{x \rightarrow \infty} \frac{D_\lam(0,x)}{\|x\|} =
\lim_{x \rightarrow \infty} \frac{\E[D_\lam(0,x)]}{\|x\|} 
= \mu_\lam(d) \qquad \text{ a.s.}
\]
where $D_\lam(x,y) := d_{G_\lam}(\ring{x},\ring{y})$ and $\ring{x}$ is the random point of the infinite cluster $C_\infty$ of $G_\lam$ that is closest to $x$
in Euclidean distance.
Informally, this means that the chemical distance 
between far away points on an infinite supercritical cluster is very likely approximately $\mu_\lam$ times the Euclidean distance. 

Our main result concerns the continuity of the time constant for continuum percolation. We prove the time constants are Lipschitz continuous in $\lam$ in the following sense.

\begin{thm}\label{Lipschitz continuity}
 Fix $d\ge 2$. For all $\lam_0>\lam_c(d)$, there exists $C(\lam_0,d)$ such that 
\[
   |\mu_{\lam}(d) - \mu_{\lam'}(d)|
   \le C(\lam_0,d)|\lam-\lam'|
\]
for each $\lam,\lam' \in [\lam_0,\infty)$.
\end{thm}

Interestingly, the proof actually yields that there is some $C(\lam_0,d)$ such that
$|\mu_\lam(d) - \mu_{\lam'}(d)| \le C(\lam_0,d)|\log(\lam/\lam')|$.
That is, the function $t \mapsto \mu_{e^t}$
is also Lipschitz on $[t_0, \infty)$
for $t_0 > \log \lam_c$.

\subsection{Proof sketch}
In order to prove Lipschitz continuity of $\mu_{\lambda}$,
it suffices to prove a bound of the form 
$|\E D_{\lam}(0,x) - \E D_{\lam'}(0,x)| \le C\|x\||\lam - \lam'|$, where $\|x\|$ is the Euclidean distance.
Here and below the constant $C$ is allowed to depend on $\lam_0$ but
not $\lam, \lam',$ or $x$.
First, we define another quantity $\widetilde{D}_{\lambda}(0,x)$
which approximates $D_\lam(0,x)$ well, so that it will suffice to prove a bound of the form
$|\E \widetilde{D}_{\lam}(0,x) - \E \widetilde{D}_{\lam'}(0,x)| \le C\|x\||\lam - \lam'|$.
This random variable $\widetilde{D}_{\lam}(0,x)$ will have the advantage that it is a bounded monotone function
of the Poisson process $X_{\lam}$ restricted to a finite volume region of $\R^d$.
This implies that $\E \widetilde{D}_{\lambda}(0,x)$ is differentiable in $\lam$
(in fact we prove a Russo-type formula for the derivative) and so it will
suffice for us to prove a bound of the form
$ \frac{d}{d\lam} \E \widetilde{D}_{\lam}(0,x) \le C \|x\|$.
The Russo-type formula will allow us to write the left-hand side as a sum over points of the Poisson process, where each summand corresponds to the effect on $\widetilde{D}_{\lam}(0,x)$ of
removing that point.
We will see that the nonzero summands are the points which lie on the geodesic $\widetilde{\pi}$ --- a sequence of points of $X_{\lam}$ which ``realize'' $\widetilde{D}_{\lam}(0,x)$.
Then, by constructing an appropriate high-probability event, we will see that this sum can be upper bounded (up to $o(\|x\|)$ error) by 
\[ \frac{1}{\lam} \E \left[ \sum_{q \in \widetilde{\pi}} d_{G_{\lam} \setminus \{q\}}(p(q), s(q)) \ind \left\{ d_{G_{\lam} \setminus \{q\}} (p(q), s(q)) < \infty \right\} \right], \]
where $p(q)$ and $s(q)$ are respectively the vertices immediately preceding and succeeding $q$ in the geodesic $\widetilde{\pi}$.
These summands have good tails
by Lemma 3.4 of Yao, Chen, and Guo \cite{YCG11},
and they enjoy a particular spatial
independence property. 
Moreover, the length of $\widetilde{\pi}$ is with
very high probability at most linear in $\|x\|$.
Therefore,
after an appropriate discretization, 
we can apply the lattice animals method of Can, Nakajima, and Nguyen \cite{CNN23}
to conclude that this sum is at most linear in $\|x\|$, 
as desired.

\section{Previous Work}

The first results on time constants
appear in the context of First Passage Percolation (FPP) on graphs, where
edges of a graph are given i.i.d. random weights, and the random metric considered is the minimum total weight of a path between two points; this metric is called the \emph{passage time}.
The chemical distance of percolation can be viewed as a special case of the passage time where all weights lie in $\{1,\infty\}$ (although many results in FPP assume
that the weights are almost surely finite, or an even stronger integrability condition).

The first regularity result for the time constant for FPP on $\Z^2$ appeared in Cox \cite{Cox80} under a uniform integrability assumption on the weight distribution. 
Still restricted to dimension two, Cox and Kesten \cite{CK81} removed the integrability condition by considering geodesics for truncated passage times. They introduced the idea of bypassing edges with weights that are too large, and established control over the length of these bypasses. Kesten \cite{Kest86} extended this result to $\Z^d, d \ge 2$ by allowing the truncation to be large enough to control the size of the clusters of closed edges. However, this technique still does not allow for edges with infinite weight. A new approach by Garet, Marchand, Procaccia, and Th\'eret  \cite{GMPT17} established continuity of the map from the underlying distribution to the time constant for general laws on $[0,+\infty]$ without any moment assumption.
Dembin \cite{Demb18} improved the regularity result to a log-Lipschitz result for a generalized percolation model in which first-passage percolation occurs in a random environment.
Building upon this work, Cerf and Dembin \cite{CD21} refined the approach by considering a multi-scale renormalization and proved Lipschitz continuity of the time constant in generalized first-passage percolation. 
The main ideas were to prove the average size of a bypass is small and to properly choose the different scales of the renormalization process. 

Keeping the same approach of looking at bypasses in ``good'' boxes under the same model and coupling, Can, Nakajima, and Nguyen \cite{CNN23} also established Lipschitz continuity of the time constant in generalized FPP. The insight here is to use lattice animals and a one-step renormalization with Russo's formula to bound the length of the detour caused by removing an edge in the geodesic.

Our proof is inspired by \cite{CNN23} but differs in a few ways. 
First, the definition of our approximation $\widetilde{D}_\lam(0,x)$
differs from that of the truncated passage time $T^{\Lambda_K}_M$
in \cite{CNN23}. In particular, $\widetilde{D}_\lam(0,x)$ only ``uses closed edges'' at the beginning and end of the path, which simplifies some technical arguments.\footnote{To be sure, many of these simplifications are possible because our setting is simpler in the sense that every edge has length 1, whereas in \cite{CNN23}, each edge has a random length.}
Next, arguments about the ``effective radius'' are replaced
by assuming uniqueness events across the whole relevant region in order to ensure connectivity of the endpoints of the geodesic even when an interior point is removed;
we then  
directly bound the length of a bypass around a deleted site assuming this robust connectivity.\footnote{These adaptations are similar to the method pursued in \cite{GorskiProcaccia} by Procaccia and the second author in the setting of graphs of polynomial growth. \cite{GorskiProcaccia} was written concurrently with this article.}
Finally,
as the space we consider is continuous, we must introduce a discretization to apply the lattice animal method in \cite{CNN23}; some care must be taken here to maintain the desired independence property when bundling points together.

\section{Preliminaries}

We recall some results for point processes and continuum percolation. 
Recall that throughout, $X_\lam$ is the Poisson point process of intensity $\lam$ on $\R^d$, and $G_\lam = G(X_\lam)$
is the associated geometric random graph.

The following result will be used repeatedly in order to bound expectations of sums indexed by points of our Poisson process.

\begin{thm}[Theorem 4.1, \cite{LP17}]\label{Palm theory} \emph{[Palm theory for Poisson processes.]}
For all non-negative bounded and measurable functions $f$ we have \begin{align*}
\E\left[\sum_{x_1,x_2,...,x_k \in X_\lam} f(x_1,x_2,...,x_k,X_\lam) \right] &= \lam^k \int_{\R^d} \E\left[f(x_1,...,x_k,X_\lam \cup \{x_1,...,x_k\}) \right]\,dx \ \\
&= \lam^k \int_{\R^d} \E^{\lam}_{x_1,...,x_k}\left[f(x_1,...,x_k,X_\lam) \right]\,dx. 
\end{align*}
\end{thm}
In the second line, $\E^\lam_{x_1,...,x_k}$ denotes
expectation with respect to the
modified probability measure\footnote{There is a sense in which $\P^\lam_{x_1,...,x_k}$ is 
the law of $X_\lam$
conditional on $x_1,...,x_k \in X_\lam$. The spatial independence of Poisson processes implies that this conditional distribution is simply equal to the law
of $X_\lam \cup \{x_1,...,x_k\}$. We do not strictly need the interpretation of $\P_{x_1,...,x_k}^\lam$ as a conditional probability,
but the notation will be convenient  in some cases.} $\P^\lam_{x_1,...,x_k}$ defined by
\begin{equation} \label{eq:Palm-distribution}
    \P^\lam_{x_1,...,x_k}(X_\lam \in A)
    := \P(X_\lam \cup \{x_1,...,x_k\} \in A).
\end{equation}

We will use a few facts about supercritical continuum percolation that Yao, Chen, and Guo established in \cite{YCG11}. The first is that it is unlikely that the origin is far from the giant component; the second is that it is unlikely to have a large component that is not part of the infinite component.

\begin{lem}[Lemma 3.3, Yao, Chen, and Guo~\cite{YCG11}] \label{Yao o close in infinite comp}
Suppose $\lambda > \lambda_c$. Then there exists a constant $c=c(\lam,d)>0$ such that for all large $r$,
\[
    \Prob(B(0,r) \cap C_\infty = \emptyset) \le \exp(-c r ^{d-1}).
\]  
\end{lem}

\begin{lem}[Lemma 3.1, Yao, Chen, and Guo~\cite{YCG11}]
\label{Yao no large component not in infinite}

Suppose that $\lambda > \lambda_c$. Then there exists a constant $c=c(\lam,d)>0$ such that, for any Borel set  $A$ in $\R^d$ and for all large $k$, we have
\[
    \Prob(\text{there exists a component } C \subset G_\lam \cap [-s/2,s/2]^d, k \le |C| < \infty, C \cap A \ne \emptyset) \le \lambda |A| \exp(-ck^{(d-1)/d}).
\]
    
\end{lem}

We also use the following local uniqueness event proven in \cite{Pen03}.

\begin{prop}[Proposition 10.13,   \cite{Pen03}]\label{largest component in metric diameter}

Let $\lam_0 >\lam_c$. Suppose ($\phi_s,s\ge 0)$ is increasing with ($\phi_s/\log s) \rightarrow \infty$ as $s \rightarrow \infty$, and with $\phi_s \le s$ for all $s$. Let $E_s$ denote the event that
at most one component of $G_\lam \cap [-s/2,s/2]^d$
has metric diameter at least $\phi_s$.
Then 
\[ \limsup_{s \rightarrow \infty} \sup_{\lam \ge \lam_0} \phi_s^{-1} \log \P^\lam[E_s^c]<0.\]
\end{prop}

\begin{remark}
    Note that 
        the above bound is uniform over $\lambda \ge \lambda_0$
        once we fix a $\lambda_0 > \lambda_c$. 
        This is a stronger statement than that found in \cite{Pen03},
        but follows from the proofs given there.
        Let us briefly sketch which details to check
        in order to get this stronger statement.
        
        The key point is that in Lemma 10.10 in \cite{Pen03},
        if we set $\mu = \lambda$,
        we can insert $\inf_{\lambda \ge \lambda_0}$
        into the quantities being bounded---in 
        between $\inf_A$ or $\inf_{A,B}$ and the ratio 
        of probabilities---and 
        the statement of the lemma will still hold
        (as long as $\lambda_0 > \lambda_c$).
        To see this, in the proof, set $\lambda' := 
        (\lambda_c + \lambda_0)/2$,
        and then the
        final bound written in the 
        proof is at least $\mathrm{(const.)}\frac{\lambda - \lambda'}{\lambda} \ge 1 - \frac{\lambda'}{\lambda_0} > 0$
        for all $\lam \ge \lam_0$.
        This establishes the stronger
        version of Lemma 10.10.
        
        Then, in the proofs of Lemmas
        10.11 and 10.12, the $\gamma$
        extracted from Lemma 10.10
        works for all $\lam \ge \lam_0$
        (and $K'$ only depends on $\lam_0$ as well).
        Thus the bounds given there
        are uniform in $\lam \ge \lam_0$.
        Thus, for the case $d \ge 3$,
        the bounds given in Proposition 10.13
        are uniform in $\lam \ge \lam_0$,
        since these are derived from Lemmas 10.11 and 10.12.
        In the $d=2$ case of
        the proof of Proposition 10.13,
        $\P^\lam(E_s)$ is bounded below 
        by crossing events, which 
        are increasing, so 
        again we get bounds which 
        are uniform in $\lam \ge \lam_0$.
\end{remark}

We also use the following fact that the chemical distance in continuum percolation has good tails. Here and throughout, we write
$x \xleftrightarrow{G} y$ to denote that two points $x$ and $y$ are connected in the graph $G$.
Recalling the definition of $\P^\lam_{0,x}$ from $\eqref{eq:Palm-distribution}$, we have:
 
\begin{lem}[Lemma 3.4, Yao, Chen, and Guo \cite{YCG11}]\label{Yao deviation of chemical distance}
Suppose that $\lam_0 > \lam_c$. 
Then there exist constants $\rho_1=\rho_1(d,\lam_0) > 0$ and $c=c(d,\lam_0) > 0$  such that, for all $x \in \R^d$, all $t \ge \rho_1 \|x\|$, and all $\lam \ge \lam_0$ we have
\begin{equation}\label{tails on distance}
    \P_{0,x}^\lam(0 \xleftrightarrow{G_\lam} x, d_{G_\lam} (0,x) \ge t) \le \exp(-ct).
\end{equation}
\end{lem}

\begin{remark}
    Again, the above statement is stronger than what is stated in \cite{YCG11} in that the
    bound holds simultaneously for all $\lam \ge \lam_0$.
    However, this stronger statement follows
    immediately from the proof in \cite{YCG11}
    if we use
    the stronger version of 
    Proposition 10.13 in \cite{Pen03} 
    which we state here as 
    \cref{largest component in metric diameter}.
\end{remark}

\section{Approximating the chemical distance}

Recall that for $v \in \R^d$, we denote by $\ring{v}$ the (random) point of the infinite cluster of $G_\lam$ which is closest in Euclidean distance to $v$, and recall that we define $D_\lam(x,y) := d_{G_\lam}(\ring{x}, \ring{y})$. Our first step is to approximate $D_\lam$ by a related random variable $\widetilde{D}_\lam$, which is a bounded monotone function of the Poisson process $X_{\lam}$ restricted to a finite volume region of $\R^d$.
Set $M = M(x) = (\log \|x\|)^2$, and define
\[
    \widetilde{D}_\lam(0, x):= \inf\{ d_{G_\lam}(p, q)+M\left(\|p\|+\|x-q\|\right): p,q \in G_\lam \}
    \cup \{M\|x\|\}.
\]
It is straightforward to check that, considered as a function of the point process $X_\lam$, $\widetilde{D}_\lam(0,x)$ is nonincreasing; that is, if $X \subset X'$, 
then $\widetilde{D}(0,x)[X] \ge \widetilde{D}(0,x)[X']$.

We now introduce notation
to keep track 
of the vertices and
edges which realize the infimum
in the definition of $\widetilde{D}_\lam(0,x)$.
If $\widetilde{D}_\lam(0,x) < M\|x\|$, define $\widetilde{0},\widetilde{x} \in G_\lam$ to be the points which satisfy
\begin{equation}
  \label{eq:tilde x}
        \widetilde{D}_\lam(0,x) = d_{G_\lam}(\widetilde{0}, \widetilde{x}) + M(\|\widetilde{0}\| + \|x - \widetilde{x}\|),
\end{equation}
and define $\widetilde{\pi}$ to be the edge path in $G_\lam$
from $\widetilde{0}$ to $\widetilde{x}$ with
$d_{G_\lam}(\widetilde{0},\widetilde{x})$ edges (if
there is more than one candidate, choose the one of
smallest total Euclidean length). If $\widetilde{D}_\lam(0,x) = M \|x\|$,
set $\widetilde{0}=\widetilde{x}=0$ and $\widetilde{\pi}$ to be the empty path from $0$ to $0$.

Note that if any vertex $v$ in $\widetilde{\pi}$ had Euclidean distance greater than $M \|x\|$ from $0$ or $x$, then we would have
\[
    \widetilde{D}_\lam(0,x) = |\widetilde{\pi}| + M(\|\widetilde{0}\| + \|x - \widetilde{x}\|) 
    \ge \|v\| + \|x-v\|> M\|x\|,
\]
which contradicts the definition of $\widetilde{D}_\lam(0,x)$. Hence we see that $\widetilde{D}_\lam(0,x)$ depends only on the Poisson process $X_\lam$ restricted to the finite volume region
$R(\widetilde{D}) := \{ y \in \R^d :
\|y\|, \|x - y\| \le M\|x\|\}.$

Our main goal in this section 
is to show that $\widetilde{D}_\lam(0,x)$
is a good approximation to $D_\lam(0,x)$ (\cref{approx of D_lam} below).
In order to do this, we first need a bound
on the tails of $D_\lam(0,x)$. 
We prove a slightly stronger statement than necessary so that we can reuse it later in the proof of \cref{geodesic is not too long}.

\begin{prop} \label{bound on tails of D}
    There exist constants $C_1,C_2, c_3$ depending only
    on $\lam_0$ and $d$ such that
    for all $t \ge C_1\|x\|$ and all $\lam \ge \lam_0$ we have
    \[
        \P^\lam(D_\lam(0,x) + M(\|\mathring{0}\|+\|x-\mathring{x}\|) \ge t) \le C_2 \exp(-c_3 t^{.99}).
    \]
\end{prop}

\begin{proof}
    Let $\rho_1(\lam_0) < \infty$ be as in \cref{Yao deviation of chemical distance}.
    Then whenever $t \ge \|x\|$ we have
    \begin{align*}
        &\P^\lam\Big(D_\lam(0,x) 
        + M\|\mathring{0}\| + M\|x - \mathring{x}\| \ge (3\rho_1 + 2)t \Big) \\
        &\le \P^\lam\left( d_{G_\lam}(\mathring{0},\mathring{x}) \ge 3 \rho_1 t, \|\mathring{0}\|, \|x - \mathring{x}\| \le \frac{t}{M}\right)
        +\P^\lam\left(\|\mathring{0}\| > \frac{t}{M}\right)
        +\P^\lam\left(\|x - \mathring{x}\| > \frac{t}{M}\right) \\
        &\le \E^\lam \left[ 
                \sum_{\substack{p \in B\left(0,\frac{t}{M}\right) \cap X_\lam \\
                                q \in B\left(x, \frac{t}{M}\right) \cap X_\lam}}
                \1\left\{\mathring{0}=p,\mathring{x}=q, 
                d_{G_\lam}(p,q) \ge 3\rho_1 t\right\}
                \right]
                + 2 \P^{\lam_0}\left(C_{\infty} \cap B\left(0, \frac{t}{M}\right)
                = \emptyset
                \right) \\
        &\le \lam^2 \int_{\substack{p \in B(0,t/M) \\ q \in B(x,t/M)}}
                \P^\lam_{p,q}\big(\mathring{0}=p,\mathring{x}=q, 
                d_{G_\lam}(p,q) \ge 3\rho_1 t\big)dp dq + 2\exp(-c_2(t/M)^{d-1}).
    \end{align*}
    (Recall the definition of $\P^\lam_{p,q}$ from \eqref{eq:Palm-distribution}.)
    In the last line we used
    \cref{Palm theory} and
    \cref{Yao o close in infinite comp}.
    Since the second summand is $O(-c_2[t/(\log t)^2]^{d-1})$, we now focus on the first term.

    Using Cauchy-Schwarz, \cref{Yao deviation of chemical distance},
    and \cref{Yao o close in infinite comp}
    then gives
    \begin{align*}
        &\lam^2 \int_{\substack{p \in B(0,t/M) \\ q \in B(x,t/M)}}
            \P^\lam_{p,q}\big(\mathring{0}=p,\mathring{x}=q, 
            d_{G_\lam}(p,q) \ge 3\rho_1 t\big)dp dq \\
        &\le
        \lam^2 \int_{\substack{p \in B(0,t/M) \\ q \in B(x,t/M)}}
            \sqrt{\P^\lam_{p,q}\big(\mathring{0}=p,\mathring{x}=q\big)
            \P^\lam_{p,q}\big(
            d_{G_\lam}(p,q) \ge 3\rho_1 t\big)}dp dq \\
        &\le \lam^2 \exp\left(-\frac{3 \rho_1 c}{2} t\right)
        \int_{\substack{p \in B(0,t/M) \\ q \in B(x,t/M)}}
        \P^\lam( C_\infty \cap B(0,\|p\|) = \emptyset)^{1/4}
        \P^\lam( C_\infty \cap B(x,\|x-q\|) = \emptyset)^{1/4} dp dq \\
        &\le
        \lam^2 \exp\left(-\frac{3 \rho_1 c}{2} t\right)
        \left(
            \int_{p \in B(0,t/M)} 
            \P^{\lam_0}\left(
            \left( \frac{\lam_0}{\lam} \right)^{1/d} C_\infty \cap
            B(0, \|p\|) = \emptyset 
        \right)^{1/4}
        dp
        \right)^2 \\
        &\le
        \lam^2 \exp\left(-\frac{3 \rho_1 c}{2} t\right)
        \left(
            \int_{p \in \R^d} 
            \P^{\lam_0}\left(
             C_\infty \cap
            B\left(0, \left(\frac{\lam}{\lam_0} \right)^{1/d} \|p\|
            \right) = \emptyset 
        \right)^{1/4}
        dp
        \right)^2 \\
        &\le
        \lam^2 \exp\left(-\frac{3 \rho_1 c}{2} t\right)
        \left(
            \int_{p \in \R^d} 
            \exp\left(-\frac{c_2}{4} \left[\frac{\lam^{1/d} \|p\|}{\lam_0^{1/d}}\right]^{d-1}\right) dp
        \right)^2 \\
        &=
        \exp\left(-\frac{3 \rho_1 c}{2} t\right)
        \left(
            \int_{R \in \R^d} 
            \exp\left(-\frac{c_2}{4\lam_0^{\frac{d-1}{d}}} \|R\|^{d-1}\right) dR
        \right)^2.
    \end{align*}
    Note that \cref{Yao deviation of chemical distance} applies
    to all $p,q$ under consideration since for such
    $p,q$ we have $\|p-q\|\le 3t$.
    In both the second and third lines Cauchy-Schwarz is used.
    In the fourth line we use that
    $X_\lam$ is equal in distribution to $\left(\frac{\lam_0}{\lam} \right)^{1/d} X_{\lam_0}$, and so in particular
    there is a coupling of $X_{\lam}$ and $X_{\lam_0}$ such
    that $C_\infty(X_{\lam})$ almost surely contains 
    $\left(\frac{\lam_0}{\lam} \right)^{1/d} C_\infty(X_{\lam_0})$.
    The cancellation of $\lam^2$ in the last line comes from
    the change of variables $R = \lam^{1/d} p$.

    Thus, if we take any $\infty > C_1 > 3 \rho_1$
    and any $0<c_3 < \frac{3 \rho_1 c}{2(3 \rho_1 + 2)}$
    we can find a $C_2$ sufficiently large that
    \[
    \P^\lam(D_\lam(0,x) + M(\|\mathring{0}\|+\|x-\mathring{x}\|) \ge t)
    \le C_2 \exp(-c_3 t^{.99}),
    \]
    as desired.
\end{proof}

We now come to the main result of this section.

\begin{prop} \label{approx of D_lam}
    For every $\lam > \lam_c$, 
    \[
    \E|D_\lam(0,x) - \widetilde{D}_\lam(0,x)| = o(\|x\|),
    \]
    (where the rate of convergence implicit in the little-o notation may depend on $\lambda$).
\end{prop}

We begin by constructing an event $A$ that happens with high probability,
and then we show an almost sure bound on $|\widetilde{D}_\lam(0,x) - D_\lam(0,x)|$ on the event $A$ via persistent use of the triangle inequality.
Recall that we defined
$R(\widetilde{D}) := \{ y \in \R^d :
\|y\|, \|x - y\| \le M\|x\|\}.$

\begin{proof}[Proof of \cref{approx of D_lam}]

For $\lam>\lam_c$, let $\rho_1$ be as in  \cref{Yao deviation of chemical distance}.

We define the following events:
\begin{align*}
        A_1 :=& 
        \left\{
        \forall
        p,q \in G_\lam \cap R(\widetilde{D}), \mbox{ either } p \stackrel{G_\lam}{\centernot\longleftrightarrow} q 
        \mbox{ or }  
        d_{G_{\lam}}(p,q) \le \rho_1 \max(\|p - q\|,M)
        \right\} \\
        A_2 :=& \{\forall v \in  G_\lam \cap R(\widetilde{D}), \mbox{ either } v \in C_\infty \mbox{ or }|C(v)| \le \|x\|/2\} \\
        A_3 :=& \{\|\ring{0}\|,\|x-\ring{x}\|\le M\},
\end{align*}
and let $A = A_1 \cap A_2 \cap A_3$. On the event $A$, all pairs which are connected in $G_\lam$ have chemical distance not too large; there are no large components other than the infinite one; and both $0$ and $x$ are not too far from the infinite component.

\begin{claim}\label{A^c is small}
    $\P(A^c) = o(M^{-2})$ (where the constants 
    implicit in the little o notation may depend on $\lambda$).
\end{claim}

\noindent \emph{Proof of \cref{A^c is small}}: 
By \cref{Yao deviation of chemical distance} 
(and \cref{Palm theory}), the probability that some pair of points in $R(\widetilde{D})$ have large chemical distance is $ O(M^{2d}\|x\|^{2d})\exp(-cM)$; by \cref{Yao no large component not in infinite}, the probability that there is a large component that is not the infinite component is $O((M\|x\|)^d\exp(-c\|x\|^{\frac{d-1}{d}})$; and by \cref{Yao o close in infinite comp}, the probability that either $0$ or $x$ is more than distance $M$ from the infinite component is $O\left(\exp(-cM^{d-1})\right)$. 
\hfill $\blacksquare$

Next, we want to show an almost-sure bound
\begin{equation} \label{target bound on A}
\1_A|\widetilde{D}_\lam(0,x)-D_\lam(0,x)| = o(\|x\|).
\end{equation}

Assume that the event $A$ holds.

Let us first take for granted that
$\widetilde{0}, \mathring{0}, \widetilde{x}, \mathring{x}$ are all connected in $G_\lam$
(we establish this in \cref{widetilde o infinite component} below).
Then the chemical distance condition guaranteed by $A_1$ shows us that:
\begin{equation} \label{other geometric bound}
\begin{aligned}
\dg(\ring{0}, \ring{x})-\dg(\widetilde{0}, \widetilde{x}) 
&\leq \dg(\ring{0}, \widetilde{0}) + \dg(\ring{x}, \widetilde{x}) \\
&\leq \rho_1\left[\max \left(\|\ring{0}-\widetilde{0}\|, M\right)+\max \left(\|\ring{x}-\widetilde{x}\|, M\right)\right] \\
& \leq \rho_1\left[\|\ring{0}\|+\|\widetilde{0}\|+\|x-\ring{x}\|+\|x-\widetilde{x}\| +2 M \right] \\
& \leq \rho_1\left[\|\widetilde{0}\|+\|x-\widetilde{x}\| +4 M \right],
\end{aligned}
\end{equation}
where the last line comes from the upper bounds on $\|\ring{0}\|,\|x-\ring{x}\|$ given by $A_3$.

To bound $\|\widetilde{0}\| + \|x-\widetilde{x}\|$, first note that
by the definition of $\widetilde{D}_\lam$ and since $\ring{0}$ and $\ring{x}$ are connected in $C_\infty$,
\[
 \widetilde{D}_\lam(0,x) = d_{G_\lam}(\widetilde{0},\widetilde{x}) + M(\|\widetilde{0}\| + \|x-\widetilde{x}\|) \le  d_{G_\lam}(\ring{0},\ring{x}) + M(\|\ring{0}\| + \|x-\ring{x}\|).
\]

Rearranging then gives
\begin{equation} \label{geometric bound}
    \begin{aligned}
    \|\widetilde{0}\| + \|x-\widetilde{x}\| \le& \frac{1}{M}[d_{G_\lam}(\ring{0},\ring{x}) - d_{G_\lam}(\widetilde{0},\widetilde{x})] + \|\ring{0}\| + \|x-\ring{x}\| \\
    \le& \frac{1}{M}[d_{G_\lam}(\ring{0},\ring{x}) - d_{G_\lam}(\widetilde{0},\widetilde{x})] + 2M.
    \end{aligned}
\end{equation}
where 
the last bound again comes from the event $A_3$.

Combining \eqref{other geometric bound} with \eqref{geometric bound} gives
\begin{align*}
\|\widetilde{0}\|+\|x-\widetilde{x}\| & \leq \frac{1}{M}\left[\dg(\ring{0}, \ring{x})-\dg(\widetilde{0}, \widetilde{x})\right]+2 M \\
& \leq \frac{\rho_1}{M}\left[\|\widetilde{0}\|+\|x-\widetilde{x}\|\right]+4 \rho_1+2 M,
\end{align*}
and rearranging gives
\begin{align*}
\|\widetilde{0}\|+\|x-\widetilde{x}\| & \leq\left(1-\frac{\rho_1}{M}\right)^{-1}(4 \rho_1+2 M) \\
& \leq 8 \rho_1+4 M.
\end{align*}

This last inequality holds whenever $\|x\|$ is sufficiently large so that $\rho_1/M < 1/2$. Plugging this last bound into
\eqref{other geometric bound} then gives
\[
\dg(\ring{0},\ring{x})-\dg(\widetilde{0}, \widetilde{x}) \leq 8 \rho_1^2+8 \rho_1 M .
\]

Note also that by \eqref{geometric bound} we have
\[
    \dg(\ring{0},\ring{x})-\dg(\widetilde{0}, \widetilde{x})
    \ge
    M( \|\widetilde{0}\| - \|\mathring{0}\| 
    + \|x - \widetilde{x}\| - \|x - \mathring{x}\|) \ge 0,
\]
where to get the second inequality 
we used the definition of
$\mathring{0},\mathring{x}$
together with the fact that
(by our connectivity assumption)
$\widetilde{0},\widetilde{x} \in C_\infty$.

Thus, putting everything together, we see that on the event $A$,
\begin{align*}
\left|D_\lam(0, x)-\widetilde{D}_\lam(0, x)\right| & = \left| \dg(\ring{0}, \ring{x})-\dg(\widetilde{0}, \widetilde{x})-M\left[\|\widetilde{0}\|+\|x-\widetilde{x}\|\right] \right| \\
& \leq \dg(\ring{0}, \ring{x})-\dg(\widetilde{0}, \widetilde{x})+M\left[\|\widetilde{0}\|+\|x-\widetilde{x}\|\right] \\
& \leq\left(8 \rho_1^2+ 8 \rho_1 M\right)+\left(8 \rho_1 M+4 M^2\right)=O\left(M^2\right)=o\left(\|x\|\right),
\end{align*}
and so we will have shown \eqref{target bound on A}
once we justify our initial assumption
that $\widetilde{0},\widetilde{x},\mathring{0},\mathring{x}$ are all connected.

\begin{claim}\label{widetilde o infinite component}
    On the event $A$,
    the points $\widetilde{0}$ and $\widetilde{x}$ lie in $C_\infty$. In particular, $\ring{0} \xlra \widetilde{0}$ and $\ring{x} \xlra \widetilde{x}$.
\end{claim}

\noindent \emph{Proof of \cref{widetilde o infinite component}:}
    On the event $A_2$, there are no large finite components intersecting $G_\lam$. Thus it suffices to show that $\widetilde{0}$ and $\widetilde{x}$ lie in a component of size at least $\|x\|/2$.

By the triangle inequality,
\[ \dg(\widetilde{0},\widetilde{x}) 
 \ge \|\widetilde{0} - \widetilde{x}\|
 \ge \|x\|-[\|\widetilde{0}\| + \|x-\widetilde{x}\|].
\]

Using \eqref{geometric bound}\footnote{Note that \eqref{geometric bound} was derived without assuming \cref{widetilde o infinite component}.} and the trivial bound $\dg(\widetilde{0},\widetilde{x}) \ge 0$,
\begin{align*}
\|\widetilde{0}\|+\|x-\widetilde{x}\| & \leq \frac{1}{M} \dg(\ring{0}, \ring{x})+2 M \\
& \leq \frac{\rho_1}{M} \|\ring{0}-\ring{x}\| + \rho_1 +2 M \\
& \leq \frac{\rho_1}{M}\left[\|x\|+\|\ring{0}\|+\|x-\ring{x}\|\right]+ \rho_1 + 2 M \\
& \leq \frac{\rho_1}{M} \|x\|+3 \rho_1+2 M .
\end{align*}
The inequality on the second line follows from the chemical distance condition given by $A_1$, 
and the fourth line comes from $A_3$.

Thus, we have
\begin{align*}
    \dg(\widetilde{0},\widetilde{x}) &\ge \|x\| - \left[\frac{\rho_1}{M}\|x\| + 3 \rho_1 + 2M \right] \\
    &=\left(1 - \frac{\rho_1}{M}\right)\|x\| - 3\rho_1 - 2M \\
    &\ge \frac{1}{2}\|x\|,
\end{align*}
where the last line holds as long as $\|x\|$ is sufficiently large. Thus $\widetilde{0}$ and $\widetilde{x}$ lie in a component of size at least $\|x\|/2$; since $A_2$ holds, they lie in the infinite component and $\widetilde{0},\ring{0},\widetilde{x}$, and $\ring{x}$ are all connected in $G_\lam$.

\hfill $\blacksquare$

Finally, we show that $A^c$ can be ignored:

\begin{claim} \label{ignore A^c}
\[
\mathbb{E}\left[\1_{A^c}\left|D_\lam(0, x)-\widetilde{D}_\lam(0, x)\right|\right] = o(\|x\|) .
\]
\end{claim}

\noindent \emph{Proof of \cref{ignore A^c}}:
By the triangle inequality and Cauchy-Schwarz, we have    
\[
\begin{aligned}
\mathbb{E} \left[\1_{A^c}|D_\lam(0, x)-\widetilde{D}_\lam(0, x)| \right]
& \leq \mathbb{E}[\1_{A^c} D_\lam(0, x)]+\mathbb{E}[\1_{A^c} \widetilde{D}_\lam(0, x)] \\
& \leq \sqrt{\mathbb{P}\left(A^c\right)}\left(\sqrt{\mathbb{E}[D_\lam(0, x)^2]}+\sqrt{\mathbb{E}\left[\widetilde{D}_\lam(0, x)^2\right]}\right) .
\end{aligned}
\]

Since, by \cref{A^c is small}, $\sqrt{\mathbb{P}\left(A^c\right)}=o(M^{-1})$, it suffices to show that 
$\mathbb{E}\left[D_\lam(0, x)^2\right], \mathbb{E}\left[\widetilde{D}_\lam(0, x)^2\right]=O(M^2\|x\|^2).$
We have a deterministic bound $\widetilde{D}_\lam(0, x)^2=O\left(M^2 \|x\|^2\right)$, so it remains to consider $\E D_\lam(0, x)^2$. 
Taking $C_1, C_2,c_3$ as in \cref{bound on tails of D},
we have
\begin{align*}
    \E[D_\lam(0,x)^2] &\le
    C_1^2 \|x\|^2 + \sum_{t=\lfloor C_1\|x\|\rfloor}^{\infty}
    (t+1)^2 \P(D_\lam(0,x) > t) \\
    &\le C_1^2\|x\|^2 + C_2\sum_{t=\lfloor C_1\|x\|\rfloor}^{\infty}
    (t+1)^2 \exp(-c_3 t^{.99}) = O(\|x\|^2),
\end{align*}
as desired. 

\hfill $\blacksquare$

Combining \eqref{target bound on A} and \cref{ignore A^c} then gives
\[
    \E[|D_\lam(0,x) - \widetilde{D}_\lam(0,x)|]
    =\E[\1_A|D_\lam(0,x) - \widetilde{D}_\lam(0,x)|]
    +\E[\1_{A^c}|D_\lam(0,x) - \widetilde{D}_\lam(0,x)|]
    = o(\|x\|),
\]
which completes our proof of \cref{approx of D_lam}.
\end{proof}

\section{Bounding the derivative by a sum along the geodesic} \label{sec:derivative bound}

We now use Russo's formula
to bound the derivative of $\E \widetilde{D}_\lam(0,x)$.
Define $f$ so that $f(X_\lam) = \widetilde{D}_\lam(0,x)$. 
For $q \in X_\lam$, define $\Delta_q f(X_\lam) = f(X_{\lambda} \setminus \{q\}) - f(X_{\lambda})$; note that for all $q$, $\Delta_q f(X_\lam) \ge 0$.
Since $f$ only depends on $X_\lam$ restricted to 
the finite-volume set $R(\widetilde{D}) := B(0,M\|x\|) \cap B(x,M\|x\|)$, Russo's formula,  \cref{Russo's formula}, tells us that
\[
\frac{d}{d\lam}\E[f(X_\lam)] =-\frac{1}{\lam}\E\left[\sum_{q\in X_\lam \cap R(\widetilde{D})}\Delta_q f(X_\lam)\right].
\]

There exist a few versions of a Russo-type formula for Poisson processes, for example see \cite{LP17}, \cite{LZ19}, \cite{Z93}, and \cite{L14}; these are related to but not identical to the form we give. We leave the details of our proof in the appendix.

Recall that we define $\widetilde{\pi}$ to be the random open path realizing $\widetilde{D}_\lam(0,x)$, that is, $\widetilde{\pi}$
is an open path in $G_\lam$ between points $\widetilde{0}, \widetilde{x}
\in G_\lam$ with
\[
    \widetilde{D}_\lam(0,x) = |\widetilde{\pi}| + M(\|\widetilde{0}\| + \|x-\widetilde{x}\|).
\]

The main result of this section is the following:
\begin{prop}\label{bounding influence by detours}
    \begin{align*}
    \E\left[ \sum_{q \in X_\lam \cap R(\widetilde{D})} \Delta_q f(X_\lam)\right]  
    &\le \E \left[\sum_{\substack{q \in \widetilde{\pi} 
    }} \1 \left\{\substack{G_\lam \setminus \{q\} \\ p(q) \leftrightarrow s(q)}\right\} d_{G_\lam \setminus \{q\}}(p(q),s(q)) \right]
    +\lam^2 o(\|x\|),
    \end{align*}
    where here $p(q)$ represents the vertex of $\widetilde{\pi}$
    preceding $q$ and $s(q)$ represents the vertex of $\widetilde{\pi}$
    following $q$.
\end{prop}

\begin{proof}

First, note that if $q \notin \widetilde{\pi}$ then we have $\Delta_q f(X_\lam) = 0$ since in this case
removing $q$ cannot increase $\widetilde{D}_\lam(0,x)$. Thus we have
\[
    \E\left[ \sum_{q \in X_\lam \cap R(\widetilde{D})} \Delta_qf(X_\lam)\right] 
    =
    \E\left[ \sum_{q \in \widetilde{\pi} } \Delta_qf(X_\lam)\right].
\]

To bound the terms $\Delta_q f(X_\lam)$,
we first define the following ``deletion-tolerant''
uniqueness event:

\[
    A' :=
    \left\{
    \begin{array}{c}
    \forall q \in X_\lam \cap R(\widetilde{D}),
    (G_\lam \setminus \{q\})\cap (q + [-\frac{M}{2},\frac{M}{2}]^d) \mbox{ has at most} \\ \mbox{one component of metric
    diameter at least } (M/2)-2
    \end{array}
    \right\}.
\]

Then we claim
\begin{claim}\label{expectation on A'}
    \begin{align*}
    \E\left[ \1_{A'}\left(\sum_{q \in \widetilde{\pi}} \Delta_q f(X_\lam)\right)\right]  &\le 
    \lam o(\|x\|) 
    + \E \left[\sum_{\substack{q \in \widetilde{\pi} \\ \setminus B(\widetilde{0}, M) \\ \setminus  B(\widetilde{x}, M)}} \1 \left\{{\substack{G_\lam \setminus \{q\} \\ p(q) \leftrightarrow s(q)}} \right\} d_{G_\lam \setminus \{q\}}(p(q),s(q)) \right],
    \end{align*}
    where here $p(q)$ represents the vertex of $\widetilde{\pi}$
    preceding $q$ and $s(q)$ represents the vertex of $\widetilde{\pi}$
    following $q$.
\end{claim}

\emph{Proof of \cref{expectation on A'}.}
   We first bound $\Delta_q f(X_\lam)$
   for $q$ near $\widetilde{0}$ and $\widetilde{x}$.
   We claim that if $q \in B(\widetilde{0},M) \cap \widetilde{\pi}$, then 
   $\Delta_q f(X_{\lambda}) \le M(M+1)$. To see this,
   consider the subpath $\pi'$ of $\widetilde{\pi}$ that starts at $s(q)$
   and continues to $\widetilde{x}$.
   We see that $|\pi'| \le |\widetilde{\pi}|$ and 
   $\|s(q)\| \le \|\widetilde{0}\| + M + 1$, so that  
   \begin{align*}
      f(X_{\lambda} \setminus \{q\})
      &\le 
      M \|s(q)\| + |\pi'| + M\|\widetilde{x}-x\| \\
      &\le
      M(M+1 + \|\widetilde{0}\|) + |\widetilde{\pi}| + M\|\widetilde{x}-x\| \\
      &=
      M(M+1) + f(X_{\lambda}),
   \end{align*}
   which gives the desired bound on $\Delta_q f(X_{\lambda})$.
   The same argument gives the same bound whenever $q$
   lies in $B(\widetilde{x},M) \cap \widetilde{\pi}$.
   Thus we have
   \begin{align*}
       \E\left[ \1_{A'}\left(\sum_{q \in \widetilde{\pi} \cap (B(\widetilde{0},M) \cup B(\widetilde{x}, M)) } 
       \Delta_q f(X_\lam)\right)\right]
       &\le
       \E[M(M+1)|X_\lam \cap (B(\widetilde{0},M) \cup B(\widetilde{x}, M))| ] \\
       &= \lam O(M^{d+2}) = \lam o(\|x\|).
   \end{align*}

   Thus, to prove \cref{expectation on A'}, it only remains to show that 
   for any $q \in \widetilde{\pi} \setminus B(\widetilde{0},M) \setminus B(\widetilde{x},M)$, we have 
   \[
      \ind_{A'} \Delta_q f(X_{\lambda})
      \le \1 \left\{{\substack{G_\lam \setminus \{q\} \\ p(q) \leftrightarrow s(q)}}\right\} d_{G_\lam \setminus \{q\}}(p(q),s(q)).
   \]
   First, we note that if $A'$ holds and $q \in \widetilde{\pi}$
   but $\|q-\widetilde{0}\|, \|q-\widetilde{x}\| > M$, then
   $p(q)$ is connected to $s(q)$ in $G_{\lambda}\setminus \{q\}$.
   This is because if they were disconnected,
   the subpath of $\widetilde{\pi}$ from $\widetilde{0}$ to $p(q)$
   and the subpath of $\widetilde{\pi}$ from $s(q)$ to $\widetilde{x}$
   would lie in two different connected components
   of $q + [-\frac{M}{2},\frac{M}{2}]^d$ which both have large diameter,
   contradicting $A'$. 

   Now, assuming $p(q)$ is connected to $s(q)$ in $G_{\lambda}\setminus \{q\}$, we see that we have a path $\pi'$ of length
   at most $|\widetilde{\pi}| + d_{G_\lam \setminus \{q\}}(p(q),s(q))$
   from $\widetilde{0}$ to $\widetilde{x}$ in $G_{\lam} \setminus \{q\}$
   obtained by composing with a geodesic from $p(q)$ to $s(q)$
   in $G_{\lambda} \setminus \{q\}$. Thus, we have the
   desired bound.

   \hfill $\blacksquare$

    The last step to showing
    \cref{bounding influence by detours}
    is to show that $A'$ is very likely,
    that is:

    \begin{claim}\label{A' event}
    \[
        \E\left[ \1_{(A')^c}
        \left(\sum_{q \in X_\lam \cap R(\widetilde{D})} \Delta_q f(X_\lam)\right)
        \right] = \lam^2 o(\|x\|).
    \]
\end{claim}

\emph{Proof of \cref{A' event}.}
Since we have almost-sure bounds $0 \le f(X_\lam) \le M\|x\|$, 
we see that 
\[ \sum_{q \in X_\lam \cap R(\widetilde{D})} \Delta_q f(X_\lam) \le M\|x\| |X_\lam \cap R(\widetilde{D})|. \] 
$|X_\lam \cap R(\widetilde{D})|$ is Poisson with parameter $\lam \mathrm{vol}(R(\widetilde{D})) =  O(\lam M^d \|x\|^d)$, and so by Cauchy-Schwarz we have
\[
        \E\left[ \1_{(A')^c}
        \left(\sum_{q \in X_\lam \cap R(\widetilde{D})} \Delta_q f(X_\lam)\right) \right]
        \le \lam \sqrt{\P((A')^c)} O(M^{d+1}\|x\|^{d+1}).
\]
To bound $\P((A')^c)$, we use \cref{Palm theory} and \cref{largest component in metric diameter}
to compute
    \begin{align*}
         \P((A')^c) &\le \E\left[\sum_{q \in R(\widetilde{D})} 
         \1  \left\{ \begin{array}{c}
    (G_\lam \setminus \{q\})\cap (q + [-\frac{M}{2},\frac{M}{2}]^d) \mbox{ has more than one} \\ \mbox{component of metric
    diameter at least } (M/2)-2
    \end{array} \right\}
         \right] \\
         &= \lam \int_{q \in R(\widetilde{D})} \P \left( \begin{array}{c}
    (G(X_\lam \cup \{q\}) \setminus \{q\})\cap (q + [-\frac{M}{2},\frac{M}{2}]^d) \mbox{ has more than one} \\ \mbox{component of metric
    diameter at least } (M/2)-2
    \end{array}
         \right)dq \\
        &= \lam \int_{q \in R(\widetilde{D})} \P \left( \begin{array}{c}
    G_\lam \cap (q + [-\frac{M}{2},\frac{M}{2}]^d) \mbox{ has more than one} \\ \mbox{component of metric
    diameter at least } (M/2)-2
    \end{array}
         \right)dq \\
          &= \lam O(M^d \|x\|^d) O(\exp(-\frac{c'}{2}M)) \\
          &= \lam o(\|x\|^{-2d-3}).
     \end{align*} 

\hfill $\blacksquare$

Combining \cref{expectation on A'} and \cref{A' event}
then gives
\begin{align*}
        \E\left[ \sum_{q \in X_\lam \cap R(\widetilde{D})} \Delta_q f(X_\lam)\right]
        &=
         \E\left[ \1_{A'} \sum_{q \in X_\lam \cap R(\widetilde{D})} \Delta_q f(X_\lam)\right] +
        \E\left[ \1_{(A')^c}\sum_{q \in X_\lam \cap R(\widetilde{D})} \Delta_q f(X_\lam)\right] \\
        &\le
        \E \left[\sum_{\substack{q \in \widetilde{\pi} \\ \setminus B(\widetilde{0}, M) \\ \setminus  B(\widetilde{x}, M)}} \1\left\{{\substack{G_\lam \setminus \{q\} \\ p(q) \leftrightarrow s(q)}}\right\} d_{G_\lam \setminus \{q\}}(p(q),s(q)) \right] + \lam^2 o(\|x\|) \\
        &\le
        \E \left[\sum_{\substack{q \in \widetilde{\pi} }} \1 \left\{{\substack{G_\lam \setminus \{q\} \\ p(q) \leftrightarrow s(q)}}\right\} d_{G_\lam \setminus \{q\}}(p(q),s(q)) \right] + \lam^2 o(\|x\|),
\end{align*}
as desired.
\end{proof}

\section{Lattice animals of dependent weight}

We now bound the expression on the right hand side of \cref{bounding influence by detours}.  
To this end,
we apply the greedy lattice animal bounds of Can, Nakajima, and Nguyen \cite{CNN23}.
Though the original statement of the following lemma is in terms of edges, there is no problem in instead considering random variables indexed by vertices.

\begin{lem}[Lemma 2.7 in 
\cite{CNN23}]\label{lattice animal bounds}
Let $(Y_{\mathbf{z}})_{\mathbf{z} \in \Z^d}$ be a family of $\mathbb{N}$-valued random variables indexed by $\Z^d$. Let $A \ge 1$. Suppose that for each $N\ge1$ and each $AN$-separated subset $S$ of $\Z^d$, the family $\{Y_\mathcal{\mathbf{z}} = N\}_{\mathbf{z} \in S}$ is independent. Define $q_N := \sup_{\mathbf{z} \in \Z^d} \P(Y_\mathbf{z} = N)$. Suppose for some $B<\infty$ we have $\sum_{N=0}^\infty N^{d+2}q_N^{1/d} \le B$. Then there exists $C$ depending only on $d, A,$ and $B$ such that the following holds. For any random path $\Pi$ in $\Z^d$ starting from $0$
such that every vertex of the path has $l^\infty$-distance 1 from the preceding vertex, and any $L \in \mathbb{N}$ we have
\[
\E[\sum_{\mathbf{z} \in \Pi}Y_\mathbf{z}] \le CL + C\sum_{\ell \ge L} \ell \P(|\Pi| = \ell)^{1/2}
\]
    
\end{lem}

The proof of \cref{lattice animal bounds} is exactly the same as the proof of Lemma 2.7 in 
\cite{CNN23}. 

In order to apply this to our setting, we will have to bound our desired quantity (which is a sum over random points in $\R^d$) by a sum over a random subset of $\Z^d$.
To do this, first, for each $\mathbf{z} \in \Z^d$, denote by $Q(\mathbf{z})$ the unit
cube $\mathbf{z} + [-\frac{1}{2},\frac{1}{2}]^d$ centered at $\mathbf{z}$.
Let $\widetilde{\Pi} \subset \Z^d$ be the set of $\mathbf{z} \in \Z^d$
such that the cube $Q(\mathbf{z})$ intersects $\widetilde{\pi}$.
Recall that for each point $q$ in the geodesic $\widetilde{\pi}$ 
we denote the point preceding $q$ in $\widetilde{\pi}$ by $p(q)$ and the point succeeding $q$ by $s(q)$. We then can write:
\begin{equation*}
    \sum_{q \in \widetilde{\pi} }  \1 \left\{{\substack{G_\lam \setminus \{q\} \\ p(q) \leftrightarrow s(q)}}\right\} d_{G_\lam \setminus \{q\}}(p(q),s(q)) 
    =
    \sum_{\mathbf{z} \in \widetilde{\Pi}}
    \sum_{\substack{q \in \widetilde{\pi} \cap Q(\mathbf{z})}} \1\left\{{\substack{G_\lam \setminus \{q\} \\ p(q) \leftrightarrow s(q)}} \right\}d_{G_\lam \setminus \{q\}}(p(q),s(q)).
\end{equation*}

Although this gives us a sum over a random path in $\Z^d$, the summands do not possess
the independence property required by \cref{lattice animal bounds}; therefore
we construct a slightly different variable
in order to employ the method of \cite{CNN23}.
For this, we have to recall
(a version of) the construction used to prove Lemma 3.4 in \cite{YCG11} (and in Antal-Pisztora \cite{AP96} to prove the main theorem).

For a scale $0< R < \infty$ and $z \in \R^d$,
consider the event
\[
    U_{z,R} :=
    \left\{
    \begin{array}{c}
    z + [-R,R]^d \mbox{ is connected to }
    z + ([-4.9R, 4.9R]^d)^c, \\
    z + [-5R, 5R] \mbox{ contains 
    at most one component of diameter 
    at least } R/5
    \end{array}
    \right\}.
\]

\begin{lem} \label{white likely}
    Let $\lam_0 > \lam_c$. Given $\epsilon > 0$, there exists $R_0=R_0(\epsilon,\lam_0,d)$ such that 
    for all $R \ge R_0$,
    \[
        \sup_{\lam \ge \lam_0}
        \Prob(U_{z,R}^c) \le \epsilon
    \]
    for every $z \in \R^d$.
\end{lem}
\begin{proof}
This follows immediately from \cref{Yao o close in infinite comp} and \cref{largest component in metric diameter}.
\end{proof}

For $\epsilon > 0$ to be fixed later,
take $R(\epsilon,\lam_0)$ as above.
Also fix an ``origin'' $\mathbf{z} \in \Z^d$.
Then, for $w \in \Z^d$, call $w$ \emph{white}
if $U_{\mathbf{z}+Rw,R}$ holds, and call $w$
\emph{black} otherwise.
Say that $w,w' \in \Z^d$ are $*$-adjacent if they have $l^\infty$-distance one, i.e.
for each $i = 1,...,d$,
$|w_i - w_i'| \le 1$.

Let $S$ be the set of $a \in \Z^d$
$*$-adjacent to $0$. 
For $a \in \Z^d$, we denote by $\mathfrak{C}_a^*$
the $*$-connected component of black sites
containing $a$,
where recall that $w$
is black if $U_{\mathbf{z} + Rw}^c$ holds.
Let $\mathcal{C}(\mathbf{z}) := \{0\} \cup \bigcup_{a \in S}  \mathfrak{C}_a^*$, let $\overline{\mathcal{C}}(\mathbf{z})$
be the set of sites either
contained in $\mathcal{C}(\mathbf{z})$ or $*$-adjacent to an element of $\mathcal{C}(\mathbf{z})$, and then take
$W(\mathbf{z}) = \bigcup_{w \in \overline{\mathcal{C}}(\mathbf{z})} (\mathbf{z} + Rw + [-5R,5R]^d)$.
The key property of this construction is the following:
\begin{lem} \label{short path}
  If $p,q \in \mathbf{z} + [-R/2,R/2]^d$
are connected in a graph $G$ which
agrees with $G_\lam$ outside 
of $\mathbf{z} + [-R,R]^d$,
then they are connected by a path in $G \cap W(\mathbf{z})$.
\end{lem}
\begin{proof}
    The argument
    is the same as that from \cite{AP96}
(see also \cite{GorskiProcaccia} for an argument using homology).
\end{proof}

We then use the following observation, which roughly states that edge-geodesics lying inside a set have length at most a constant times the volume of the set.

\begin{prop} \label{path volume bound}
    Given $d \ge 2$, there exists $C_d < \infty$ such that the following holds.
    If $G$ is an induced subgraph of $G_\lam$, for any $W \subset \R^d$ 
    and any path $\gamma$ in $G$, 
    there exists a self-avoiding path $\gamma'$ in $G$
    with the same endpoints as $\gamma$
    such that all vertices of $\gamma'$ lie in $\gamma$, and
    \[
        |\gamma' \cap W| \le C_d \mathrm{vol}(W + B(1/2)).
    \]
    Moreover,
    we can either take $\gamma = \gamma'$
    or $|\gamma'| < |\gamma|$.
\end{prop}
\begin{proof}
    First note that we can construct
    $\gamma'$ with the same endpoints
    as $\gamma$ with the property
    that any pair of vertices in $\gamma' \cap W$
    which are not adjacent in $\gamma'$
    have distance larger than $1$; we do so as follows.
    If $\gamma$ already has this
    property, take $\gamma' = \gamma$.
    Otherwise, take $\gamma$,
    choose a pair of
    non-adjacent vertices $x,y \in \gamma$
    with $\|x - y\| \le 1$, and replace the
    subpath of $\gamma$ from the first instance of $x$ to the last instance of $y$
    with the edge $e=\{x,y\}$
    (which is in $G$).
    This modification strictly decreases the length of $\gamma$, and it decreases the number of pairs
    of non-adjacent vertices with distance at most one by at least one;
    it also does not change the endpoints
    of $\gamma$.
    After inductively applying this
    procedure finitely many times,
    we will have $\gamma'$ with the desired property.

    Next, there
    exists a 1-separated vertex subset 
    $S \subset W \cap \gamma'$ 
    such that $|S| \ge \frac{1}{2}|W \cap \gamma'|$;
    simply take $S$ to be a subset of maximal size among subsets which contain 
    no pair of vertices adjacent in $\gamma'$.

Then we have the inclusion of disjoint balls
\[
\bigsqcup_{s \in S} B\left(s,\frac{1}{2}\right) \subseteq W+B\left(0,\frac{1}{2}\right) 
\]
which implies the volume bound
\[
    |S| \mathrm{vol}\left(B_d\left(1/2\right)\right) \le \mathrm{vol}(W+B(1/2)).
\]
So we have
\[
|\gamma' \cap W| \le 2 |S| \le  \frac{2}{\mathrm{vol}(B_d(1/2))} \mathrm{vol}(W+B(1/2)),\]
and then we can take $C_d := 2/\mathrm{vol}(B_d(1/2))$, which depends only on $d$, and we are done.
\end{proof}

Now we can use all of the above to construct a variable $Y_\mathbf{z}$ which satisfies the hypotheses of \cref{lattice animal bounds} and which bounds our desired quantity up to a constant factor.

\begin{prop} \label{justify Yz}
   Let $R$ be fixed.
   For each $\mathbf{z} \in \Z^d$, define the random variable
   \[
      Y_{\mathbf{z}} := 
      |\overline{\mathcal{C}}(\mathbf{z})|
   \]
   where $\overline{\mathcal{C}}(\mathbf{z})$
   is as defined above.
   Then there exists some constant $C'_{d,R} < \infty$
   depending only on $d$ and $R$ such that 
\[
\sum_{\substack{q \in \widetilde{\pi} }} \1 \left\{{\substack{G_\lam \setminus \{q\} \\ p(q) \leftrightarrow s(q)}}\right\} d_{G_\lam \setminus \{q\}}(p(q),s(q)) 
\le
C'_{d,R}\sum_{\mathbf{z} \in \widetilde{\Pi}}
Y_{\mathbf{z}}.
\]
\end{prop}


\begin{proof}
    To avoid clutter, for three points $p,q,s \in G_\lam$ let us define
    \[
       Z_{p,q,s} := \1 \left\{\substack{G_\lam \setminus \{q\} \\ p \leftrightarrow s}\right\} d_{G_\lam \setminus \{q\}}(p,s).
    \]
    Then the quantity we want to bound is
    \begin{align*}
    \sum_{\substack{q \in \widetilde{\pi}} } Z_{p(q),q,s(q)} 
    &=
    \sum_{\mathbf{z} \in \widetilde{\Pi}}
    \sum_{\substack{q \in \widetilde{\pi} \cap Q(\mathbf{z})}} Z_{p(q),q,s(q)} 
    \end{align*}
    By \cref{short path},
    if $p,s$ adjacent to $q \in Q(\mathbf{z})$ are connected in $G_\lam \setminus \{q\}$, then they are connected by a path $\gamma$ in $W(\mathbf{z})$.
    By \cref{path volume bound},
    we then have a path $\gamma'$
    from $p$ to $s$ in $G_\lam \setminus \{q\} \cap W$ with
    \begin{align*}
        |\gamma'|
        &\le C_d \mathrm{vol}(W + B_d(1/2)) \\
        &\le C_d |\overline{\mathcal{C}}(\mathbf{z})|(10R+1)^d.
    \end{align*}
    This gives a bound for each $Z_{p,q,s}$,
    and therefore we have
    \begin{align*}
        \sum_{\substack{q \in \widetilde{\pi}} } Z_{p(q),q,s(q)} 
    &\le
    \sum_{\mathbf{z} \in \widetilde{\Pi}}
    \sum_{\substack{q \in \widetilde{\pi} \cap Q(\mathbf{z})}} C_d(10R+1)^d |Y_\mathbf{z}| \\
    &= 
    C_d (10R+1)^d 
    \sum_{\mathbf{z} \in \widetilde{\Pi}} |\widetilde{\pi} \cap Q(\mathbf{z})|
    |Y_\mathbf{z}| .
    \end{align*}
    But since $\widetilde{\pi}$
    is a minimal edge path from $\widetilde{0}$
    to $\widetilde{x}$,
    \cref{path volume bound}
    applied to $\widetilde{\pi}$
    shows that 
    \[
        |\widetilde{\pi} \cap Q(\mathbf{z})|
        \le C_d \mathrm{vol}(Q(\mathbf{z}) + B_d(1/2))
        \le 2^d C_d.
    \]
    Combining all of the above then gives
    \begin{align*}
        \sum_{\substack{q \in \widetilde{\pi}} } Z_{p(q),q,s(q)} 
        \le
        2^d C_d^2 (10R + 1)^d \sum_{\mathbf{z} \in \widetilde{\Pi}} |Y_{\mathbf{z}}|,
    \end{align*}
    and we are done.
\end{proof}

Now we establish that the family $(Y_{\mathbf{z}})_{\mathbf{z} \in \Z^d}$
satisfies the hypotheses of \cref{lattice animal bounds}. 

\begin{lem} \label{independence}
    For each $R< \infty, d\ge 2$ there exists $A<\infty$ depending only on $d$ and $R$ such that the following holds.
    For any $N \ge 1$ and any $AN$-separated subset $S \subset \Z^d$, the family of events
    $\{Y_{\mathbf{z}} = N\}_{\mathbf{z} \in S}$ is independent.
\end{lem}
\begin{proof}
    For a fixed $R$,
    the event $\{Y_\mathbf{z} = |\overline{\mathcal{C}}(\mathbf{z})| = N\}$
    is witnessed by a 
    connected set in $\Z^d$
    containing $0$
    with size at most $N$,
    and the state of each site $w \in \Z^d$
    is determined by 
    the restriction of $X_\lam$ to $\mathbf{z} + Rw  + [-5R,5R]^d$.
    Therefore this event
    depends only on $X_\lam$
    restricted to $B(\mathbf{z},10dRN)$.


    Therefore, if $S \subset \Z^d$ is a $(30dR N)$-separated subset,
    the events $\{Y_{\mathbf{z}} = N\}_{\mathbf{z} \in S}$ all 
    depend on $X_\lam$ restricted to pairwise disjoint regions, and hence
    are independent. Taking $A=30dR$, we are done.
\end{proof}

\begin{lem} \label{Y moment bound}
    Let $\lam_0 > \lam_c$. Then there exist
    $R < \infty$ and $B < \infty$ depending
    only on $d$ and $\lam_0$ such that
    the following holds.
    If we define $Y_{\mathbf{z}}$
    as in \cref{justify Yz}
    with parameter $R$, then
    for all $\lam \ge \lam_0$ we have
    \[
        \sum_{N=0}^\infty N^{d+2} q_N^{1/d} \le B,
    \]
    where
    \[
        q_N := \sup_{\mathbf{z} \in \Z^d} 
        \P^\lam(Y_{\mathbf{z}} = N).
    \]
\end{lem}
\begin{proof}
First let us establish a bound on the tail of $Y_{\mathbf{z}}$ for a sufficiently large choice of $R$.

First, set $\delta = 1/(4 \cdot 3^d)$,
so in particular $0 < \delta < 1/(2 \cdot 3^d)$.
Then, using domination by product measures (\cite{LSS1997}, Theorem 0.0.(ii)),
choose $\epsilon > 0$
sufficiently small that any $11d$-dependent field on $\Z^d$ with $\Prob(w \mbox{ is white}) \ge 1 - \epsilon$
dominates an independent site percolation
on $\Z^d$ with $\Prob(w \mbox{ is white}) = 1-\delta$.
We then choose $R(\epsilon,\lam_0,d)$
as in \cref{white likely}
and perform the construction of $\overline{\mathcal{C}}(\mathbf{z})$ with parameter $R$.
In particular, we note that
by our choice of $R$,
for any $a \in \Z^d$, $|\mathfrak{C}_a^*|$
is stochastically dominated by the law of the size of the $*$-connected cluster of black sites at the origin of a Bernoulli site percolation with $\Prob(0 \mbox{ is black}) = \delta$.

Now, to bound the tail of $Y_{\mathbf{z}}$, note that
\begin{align*}
    Y_\mathbf{z} =
    |\overline{\mathcal{C}}(\mathbf{z})|
    \le 3^d |\mathcal{C}(\mathbf{z})| 
    \le 3^d(1 + \sum_{a \in S} |\mathfrak{C}_a^*|),
\end{align*}
so a union bound gives
\begin{align*}
    \Prob(Y_\mathbf{z} \ge N)
    &\le  
    \Prob( \sum_{a \in S} |\mathfrak{C}_a^*| \ge (N/3^d) - 1 ) \\
    &\le \sum_{a \in S} 
    \Prob( |\mathfrak{C}_a^*| \ge (N/3^{2d}) - 3^{-d} ) \\
    &\le 3^d \sum_{k=\lceil (N/3^{2d}) - 1 \rceil}^\infty [(2 \cdot 3^d) \delta]^k,
\end{align*}
where the last bound estimates the tail of the size of an independent Bernoulli site percolation cluster using the fact that the number of $*$-connected subsets of $\Z^d$
of size $k$ containing the origin is at most $(2 \cdot 3^d)^k$. Now, by our choice of $\delta$, this is decaying exponentially in $N$; let us say
\[
    \Prob(Y_\mathbf{z} \ge N)
    \le C e^{-c N},
\]
where $C,c$ now only depend on $d$.

Hence we compute
\begin{align*}
        \sum_{N=0}^\infty N^{d+2} q_N^{1/d} &\le 
        C \sum_{N=0}^{\infty} N^{d+2} e^{-(c/d)N}
        \le B,
\end{align*}
where $B$ is a function of $C, c,$ and $d$, and hence
actually only depends on $d$.
\end{proof}

Finally, to make use of
\cref{lattice animal bounds} effectively
we must bound the tails of 
$|\widetilde{\Pi}|$.
As a technical point, in order to apply
\cref{lattice animal bounds}, we need a random path
starting from $0$, so we define a new 
random path $\widehat{\Pi}$
by $\gamma \cup \widetilde{\Pi}$,
where $\gamma$ is a shortest path in $\Z^d$
from $0$ to the first point of $\widetilde{\Pi}$.
\begin{lem} \label{geodesic is not too long}
    There exist constants $C_1,C_2, c_3$ depending only
    on $\lam_0$ and $d$ such that
    for all $t \ge C_1\|x\|$ we have
    \[
        \P(\widehat{\Pi} \ge t) \le C_2 \exp(-c_3 t^{.99}).
    \]
\end{lem}
\begin{proof}
    First, we can use the crude bound
    \begin{align*}
        |\widetilde{\Pi}| \le |\widetilde{\pi}|
    \end{align*}
    and
    \begin{align*}
        |\gamma| \le \sqrt{d}(\|\widetilde{0}\|+1)
    \end{align*}
    to get (assuming that $\|x\|$ and hence $M$ is sufficiently large)
    \begin{align*}
        |\widehat{\Pi}| \le \sqrt{d}(\|\widetilde{0}\|+1)
        +|\widetilde{\pi}|
        \le M\|\widetilde{0}\| + |\widetilde{\pi}|
        +M\|x-\widetilde{x}\| + \sqrt{d} = \widetilde{D}_\lam(0,x) +\sqrt{d}.
    \end{align*}
    Thus, it suffices to bound the tails of $\widetilde{D}_\lam(0,x)$.
    But since we also have the bound
    $\widetilde{D}_\lam(0,x) \le D_\lam(0,x) + M(\|\ring{0}\|+\|x-\ring{x}\|)$,
    our desired statement now follows directly from 
    \cref{bound on tails of D}.
\end{proof}

Finally we can get the desired linear bound on our 
random sum:
\begin{lem} \label{linear bound on sum}
    Let $\lam_0 > \lam_c$. 
    Choose $R$ as in \cref{Y moment bound} and define
    $Y_{\mathbf{z}}$ as in 
    \cref{justify Yz}
    with parameter $R$.
    Then there exists
    $B' < \infty$ depending only on $\lam_0$ and $d$
    such that for every $\lam \ge \lam_0$ we have
    \begin{align*}
        \E^{\lam}\left[ 
        \sum_{\mathbf{z} \in \widetilde{\Pi}} Y_{\mathbf{z}}
        \right]
        \le B' \|x\|
    \end{align*}
    whenever $\|x\| \ge 1$.
\end{lem}
\begin{proof}
    Since the summands are all nonnegative
    and by construction $\widetilde{\Pi} \subset \widehat{\Pi}$,
    it suffices to show the same bound
    with $\widetilde{\Pi}$ replaced by $\widehat{\Pi}$.

    Choosing $R$ and $B$ based on $\lam_0$ and $d$ as in \cref{Y moment bound},
    \cref{independence} provides us with a parameter $A$
    so that the family 
    $\{Y_{\mathbf{z}}\}_{\mathbf{z} \in \Z^d}$
    satisfies the hypotheses of \cref{lattice animal bounds}.
    Taking $C_1$ as in \cref{geodesic is not too long}
    and setting $L = C_1 \|x\|$, applying \cref{lattice animal bounds} then gives
    \begin{align*}
        \E^{\lam}\left[ 
        \sum_{\mathbf{z} \in \widetilde{\Pi}} Y_{\mathbf{z}}
        \right]
        \le
        C\left(
        C_1 \|x\|
        + C_2 \sum_{t= \lfloor C_1 \|x\| \rfloor}^{\infty}
        \exp(-c_3 t^{.99})
        \right),
    \end{align*}
    for some $C$ only depending on $\lam_0$ and $d$ (through $A$ and $B$).
    Then we can choose $B'$ depending only 
    on $C, C_1, C_2, c_3$ (hence only on $\lam_0$ and $d$)
    such that the above is bounded by
    $B'\|x\|$ whenever $\|x\| \ge 1$.
\end{proof}

\section{Lipschitz continuity of the time constant}

Finally, we can prove \cref{Lipschitz continuity}.

\begin{proof} [Proof of \cref{Lipschitz continuity}]

First let us prove a Lipschitz bound for $\widetilde{D}_\lam$. 
Recall that,
since $\widetilde{D}_\lam(0,x)$
is bounded and only depends on $X_\lam$ restricted 
to the finite volume region $R(\widetilde{D}) \subset \R^d$,
its expectation is differentiable by \cref{Russo's formula}.
Fixing $\lam_0>\lam_c$,
for any $\lam_0 \le \lam < \lam ' < \infty$, we then have

\begin{align*}
    0 &\le \E \widetilde{D}_\lam(0,x) - \E \widetilde{D}_{\lam'}(0,x) \\
    &=\int^{\lam'}_\lam - \frac{d}{d\ell} \E \widetilde{D}_\ell(0,x)d\ell \\
    &= \int^{\lam'}_\lam \frac{1}{\ell}\E\left[\sum_{q \in X_\ell \cap R(\widetilde{D})} \Delta_q f(X_\ell) \right] d\ell \\
    &\le \int^{\lam'}_\lam \frac{1}{\ell}  \E\left[\sum_{q \in \widetilde{\pi}} \1 \left\{{\substack{G_\ell \setminus \{q\} \\ p(q) \leftrightarrow s(q)}}\right\} d_{G_\ell \setminus \{q\}} (p(q),s(q)) \right] d\ell + 
    \left( \int_{\lam}^{\lam'} \ell d\ell \right) o(\|x\|)\\
    &\le \int^{\lam'}_{\lam} \frac{1}{\ell} C'_{d,R} \E^{\ell}\left[\sum_{\mathbf{z} \in \widetilde{\Pi}} Y_\mathbf{z}\right] d\ell + (\lam')^2 o(\|x\|) \\
    &\le (C_{d,R}'/\lam_0) B' \|x\| |\lam' - \lam| + (\lam')^2 o(\|x\|).
\end{align*}

The equality in the  third line is from \cref{Russo's formula}
(recalling the definitions of $f$ and $\Delta_q$
from the beginning of \cref{sec:derivative bound}); 
the fourth line is from \cref{bounding influence by detours}.
In the fifth line,
we have chosen $R$ based on $\lam_0$ and $d$ as specified by \cref{Y moment bound},
and defined $Y_{\mathbf{z}}$
with parameter $R$
as in \cref{justify Yz};
the fifth line then follows from \cref{justify Yz},
and the final bound from \cref{linear bound on sum}.

Therefore we have

\begin{align*}
    0 \le \mu_\lam - \mu_{\lam'} &= \lim_{x \rightarrow \infty} \frac{\E D_{\lam} (0,x) - \E D_{\lam'}(0,x)}{\|x\|} \\
    &= \lim_{x \rightarrow \infty} \frac{\E \widetilde{D}_\lam (0,x) - \E \widetilde{D}_{\lam'}(0,x)}{\|x\|} \\
    &\le \limsup_{x \rightarrow \infty}
    \frac{(C_{d,R}'/\lam_0) B' |\lam' - \lam| \|x\| + (\lam')^2o(\|x\|)}{\|x\|} \\
    &= (C_{d,R}'/\lam_0) B' |\lam' - \lam|. 
\end{align*}
The second line is from \cref{approx of D_lam},
and the third line is the last inequality we derived.
Thus, the function $\lam \mapsto \mu_\lam$
is Lipschitz on $[\lam_0, \infty)$ with constant $(C_{d,R}'/\lam_0) B'$.

Note also that, if instead of using the bound $(1/\ell) \le (1/\lam_0)$ above, we compute the integral, we obtain 
a bound of
$C'_{d,R} B' \log(\lam'/\lam)$.
Hence $\mu_{\lam}$
is Lipschitz as a function of \emph{the logarithm of $\lam$},
that is,
$t \mapsto \mu_{e^t}$
is Lipschitz in $t$ on $[t_0,\infty)$
for $t_0 > \log \lam_c$.
\end{proof}

\section*{Acknowledgments}
K.D. is partially supported by NSF grant DMS-2246624. 

\bibliographystyle{abbrv}
\bibliography{refs}

\appendix
\section{A Russo-type formula} \label{appendix}

Throughout this section, 
denote by $X_{\lambda}$ a Poisson point process of intensity $\lambda$
on $K$, where $K$ is a finite volume 
Borel subset of $\R^d$,
and let 
$f: \{ F \subset K : F \mbox{ finite } \} \to \R$
be a bounded measurable function.

\begin{prop}[A Russo-type formula]\label{Russo's formula} 
   For every $\lambda > 0$,
   $\frac{d}{d\lambda} \E f(X_{\lambda})$ exists
   and we have
   \[\frac{d}{d\lambda} \E f(X_{\lambda}) = 
   \frac{1}{\lambda}
   \E\left[ \sum_{x \in X_{\lambda}} [f(X_{\lambda}) - f(X_{\lambda} \setminus \{x\})] \right].
   \]
\end{prop}
\begin{proof}
    First, we show that
    $\E f(X_{\lambda})$ is
    differentiable in $\lambda$.
    To see this, first note
    that, since $K$ has finite
    volume, $X_{\lambda}$ can
    be sampled by first sampling
    $N \sim \mathrm{Poi}(\lambda |K|)$
    and then letting
    $X_{\lambda} = U_{[N]}$,
    where $U_1,U_2,...$
    is a sequence of independent random variables sampled according to the uniform measure on $K$ and $U_{[k]} := \bigcup_{i=1}^k \{U_i\}$.
    Therefore we can write
    \begin{align*}
    \E f(X_\lambda)
    &= \sum_{k = 0}^{\infty} \E[f(X_{\lambda}) | N = k] 
    \Prob( N = k) \\ 
    &= \sum_{k=0}^{\infty}
    \E[f(U_{[k]})] 
    e^{- \lambda |K|}
    \frac{(\lambda |K|)^k}{k!}.
    \end{align*}

    Note that $\E[f(U_{[k]})]$
    is independent of $\lambda$,
    so each term of this series 
    is differentiable in $\lambda$.
    Moreover, 
    (since $\E f(U_{[k]}) \le \|f\|_{\infty}$) it is
    straightforward to check that the
    sum of the termwise derivatives
    converges uniformly on 
    compact subsets of $[0,\infty)$;
    therefore $\E f(X_\lambda)$
    is differentiable on
    $(0,\infty)$ as desired.

    Now recall the following coupling for $X_\lam$ and $X_{\lam - \epsilon}$:
    Given $\eta \in (0,1)$, let
    $X_{\lambda, \eta}$
    be the random subset of $X_{\lambda}$ where each
    element of $X_{\lambda}$ is independently
    retained with probability $\eta$
    and deleted with probability 
    $1 - \eta$.
    Then $X_{\lambda, \eta}$
    is a Poisson point process
    on $K$ with intensity
    $\eta \lambda$
    and 
    $X_{\lambda} \setminus 
    X_{\lambda, \eta}$
    is a Poisson point process
    on $K$ with intensity
    $(1 - \eta)\lambda$.
    In particular,
    if we take 
    $\eta = \frac{\epsilon}{\lambda}$, we see that
    $X_{\lambda} \setminus 
    X_{\lambda, \frac{\epsilon}{\lambda}}$ is equal
    in distribution 
    $X_{\lambda - \epsilon}$.
    Therefore we have
    \begin{align*}
        \frac{d}{d\lambda}
        \E f(X_{\lambda})
        &=
        \lim_{\epsilon \to 0^+}
        \frac{\E f(X_{\lambda}) 
        - \E f(X_{\lambda - \epsilon})}
        {\epsilon} \\ 
        &= 
        \lim_{\epsilon \to 0^+}
        \frac{\E[ f(X_{\lambda}) 
        - f(X_{\lambda} 
        \setminus X_{\lambda,\frac{\epsilon}{\lambda}})]}
        {\epsilon}.
    \end{align*}
    
    Next, note that
    \begin{align*}
       \E \left[
       |f(X_{\lambda}) - 
       f(X_{\lambda} \setminus X_{\lambda, \frac{\epsilon}{\lambda}})| 
       \1_{\{|X_{\lambda, \frac{\epsilon}{\lambda}}| \ge 2\}}
       \right]
       \le& 2 \|f\|_{\infty} \Prob( |X_{\lambda, \frac{\epsilon}{\lambda}}| \ge 2) \\
       =& 2 \|f\|_{\infty} \Prob( |X_{\epsilon}| \ge 2) \\
       =& O(\epsilon^2).
    \end{align*}


    Therefore, in order to compute the desired limit,
    it suffices to consider
    \[
    \E \left[ [f(X_{\lambda}) - f(X_{\lambda} \setminus X_{\lambda, \frac{\epsilon}{\lambda}})] 
    \ind_{\{ |X_{\lambda, \frac{\epsilon}{\lambda}}|
    \le 1\}} \right] =
    \E \left[ [f(X_{\lambda}) - f(X_{\lambda} \setminus X_{\lambda, \frac{\epsilon}{\lambda}})] 
    \ind_{\{ |X_{\lambda, \frac{\epsilon}{\lambda}}|
    = 1 \}} \right],
    \]
    where for the above equality we used the fact
    that $f(X_{\lambda}) - f(X_{\lambda} \setminus X_{\lambda, \frac{\epsilon}{\lambda}}) = 0$
    on the event $\left\{|X_{\lambda, \frac{\epsilon}{\lambda}}| = 0 \right\}$.

    Now, conditioning on $X_{\lambda}$ we can compute
    \begin{align*}
        \E \left[ [f(X_{\lambda}) - f(X_{\lambda} \setminus X_{\lambda, \frac{\epsilon}{\lambda}})] 
        \ind_{\{ |X_{\lambda, \frac{\epsilon}{\lambda}}|
        = 1 \}} \middle| X_{\lambda} \right]
        &= 
        \sum_{S \subset X_{\lambda}}
        (f(X_{\lambda}) - f(X_{\lambda} \setminus S))
        \ind_{\{ |S| = 1 \}} \Prob( X_{\lambda, \frac{\epsilon}{\lambda}} = S | X_{\lambda}) \\
        &=
        \left(\frac{\epsilon}{\lambda} \right)
        \left(1 - \frac{\epsilon}{\lambda}\right)^{|X_{\lambda}| - 1}
        \sum_{x \in X_{\lambda}} 
        [f(X_{\lambda}) - f(X_{\lambda} \setminus \{x\})].
    \end{align*}

    Thus, we have
    \begin{align*}
        \frac{d}{d\lambda} \E f(X_{\lambda}) &=
        \lim_{\epsilon \to 0^+}
        \frac{1}{\epsilon} 
        \E \left[ [f(X_{\lambda}) - f(X_{\lambda} \setminus X_{\lambda, \frac{\epsilon}{\lambda}})] 
        \ind_{\{ |X_{\lambda, \frac{\epsilon}{\lambda}}|
        = 1 \}} \right] \\
        &=
        \lim_{\epsilon \to 0^+}
        \frac{1}{\lambda}
        \E \left[
        \left(1 - \frac{\epsilon}{\lambda}\right)^{|X_{\lambda}| - 1}
        \sum_{x \in X_{\lambda}} 
        [f(X_{\lambda}) - f(X_{\lambda} \setminus \{x\})]
        \right] \\
        &=
        \frac{1}{\lambda}
        \E \left[
        \sum_{x \in X_{\lambda}} 
        [f(X_{\lambda}) - f(X_{\lambda} \setminus \{x\})]
        \right],
    \end{align*}
    where the last equality follows from the
    Dominated Convergence Theorem,
    since the random variables being integrated
    are dominated by the integrable
    random variable $4\|f\|_{\infty}|X_\lam|$
    for $0 < \epsilon < \lam/2$.
\end{proof}

\end{document}